\documentclass{article}
\usepackage{amsfonts}
\usepackage{amssymb, amsmath, amsthm}
\usepackage{algorithm}
\usepackage{algorithmic}
\usepackage[letterpaper,body={13.5cm,22cm}, mag=1000]{geometry}
\usepackage{amssymb}
\usepackage{secdot}
\usepackage{setspace}
\usepackage{titlesec}
\titleformat{\section}[runin]{\bfseries\filcenter}{\thesection}{1em}{}

\renewcommand{\thesection}{\arabic{section}}
\title{\large \bf On automorphisms of finite $p$-groups}

\author{\small \bf Hemant Kalra and Deepak Gumber   \\
\small \em School of Mathematics \\
\small \em Thapar Institute of Engineering and Technology, Patiala - 147 004,
India\\
\small \em emails: happykalra26@gmail.com, dkgumber@gmail.com\\
 }

\date{}
\newtheorem{thm}{Theorem}[section]
\newtheorem{lm}[thm]{Lemma}

\newtheorem{cor}[thm]{Corollary}
\newtheorem{rem}[thm]{Remark}
\newtheorem{expl}[thm]{Example}

\DeclareMathOperator{\Aut}{Aut}
\DeclareMathOperator{\inn}{Inn}

\DeclareMathOperator{\Hom}{Hom}
\DeclareMathOperator{\cl}{cl}

\begin{document}

\maketitle
\begin{abstract}
It is proved in [J. Group Theory, {\bf 10} (2007), 859-866] that if $G$ is a  finite $p$-group such that $(G,Z(G))$ is a Camina pair, then $|G|$ divides $|\Aut(G)|$. We give a very  short and elementary proof of this result.
\end{abstract}
\vspace{2ex}


\vspace{2ex}


\section{Introduction}

Let $G$ be a finite non-abelian $p$-group. The problem ``Does the order, if it is greater than $p^2$, of a finite non-cyclic $p$-group divide the order of its automorphism group?'' is a well-known problem \cite[Problem 12.77]{maz} in finite group theory. Gasch$\ddot{\mbox{u}}$tz \cite{gas} proved that any finite $p$-group of order at least $p^2$ admits a non-inner automorphism of order a power of $p$. It follows that the problem has an affirmative answer for finite $p$-groups with center of order $p$. This immediately answers the problem positively for  finite $p$-groups of maximal class. Otto \cite{ott} also gave an independent proof of this result. Fouladi et al. \cite{fou} gave a supportive answer to the problem for finite $p$-groups of co-class 2. For more details on this problem, one can see the introduction in the paper of Yadav \cite{yad1}. In \cite[Theorem A]{yad1}, Yadav proved that if $G$ is a finite $p$-group such that $(G,Z(G))$ is a Camina pair, then $|G|$ divides $|\Aut(G)|$.  He also proved the important result \cite[Corollary 4.4]{yad1} that the group of all class-preserving outer automorphisms is non-trivial for finite $p$-groups $G$ with $(G,Z(G))$  a Camina pair.

In this paper, we give different and very short proofs of these results of Yadav using elementary arguments.

Let $G$ be a finite $p$-group. Then $(G,Z(G))$ is called a Camina pair if $xZ(G) \subseteq x^G$ for all $x\in G-Z(G)$, where $x^G$ denotes the conjugacy class of $x$ in $G$. In particular, if $(G,G')$ is a Camina pair, then $G$ is called a Camina $p$-group.

\section{Proofs}

We shall need the following lemma which is a simple modification of a lemma of Alperin \cite[Lemma 3]{alp}.

\begin{lm}
Let $G$ be any group and $B$ be a central subgroup of $G$ contained in a normal subgroup $A$ of $G$. Then the group $\Aut_{A}^{B}(G)$of all automorphisms of $G$ that induce the identity on both $A$ and $G/B$ is isomorphic onto
$\mathrm{Hom}(G/A,B)$.
\end{lm}

\begin{thm}
Let $G$ be a finite $p$-group such that $(G,Z(G))$ is a Camina pair. Then $|G|$ divides $|\Aut(G)|.$
\end{thm}

\begin{proof}  Observe that  $Z(G)\le G'\le \Phi(G)$  and, therefore,  $Z(G)\le Z(M)$ for every maximal subgroup $M$ of $G$. Suppose that $Z(G)<Z(M_1)$ for some  maximal subgroup $M_1$ of $G$.  Let $G=M_1\langle g_1\rangle $, where $g_1\in G-M_1$ and $g_{1}^{p}\in M_1$. Let $g\in Z(M_1)-Z(G)$. Then  
\[|Z(G)|\le |[g, G]|= |[g, M_1\langle g_1\rangle]|=|[g,\langle g_1\rangle]|\le p\] 
implies that $|Z(G)|=p.$ The result therefore follows by Gasch$\ddot{\mbox{u}}$tz \cite{gas}. We therefore suppose that $Z(G)=Z(M)$ for every maximal subgroup $M$ of $G$. We prove that $C_G(M)\le M$. Assume that there exists $g_0\in C_G(M_0)-M_0$ for some maximal subgroup $M_0$ of $G$. Then $G=M_0\langle g_0\rangle$ and thus $g_0\in Z(G)$, because $g_0$ commutes with $M_0$. This is a contradiction because $Z(G)\le \Phi(G)$. Therefore $C_G(M)\le M$ for every maximal subgroup $M$ of $G$. Consider the group $\Aut_{M}^{Z(G)}(G)$ which  is isomorphic to $\Hom(G/M,Z(G))$ by Lemma 2.1. It follows that $\Aut_{M}^{Z(G)}(G)$ is non-trivial. Let $\alpha\in \Aut_{M}^{Z(G)}(G)\cap (\inn(G))$. Then $\alpha$ is an inner automorphism induced by some $g\in C_G(M)=Z(M)$. Since $Z(G)=Z(M)$,  $\alpha$ is trivial. It follows that 
$$|(\Aut_{M}^{Z(G)}(G))(\inn(G))|=|(\Aut_{M}^{Z(G)}(G))||(\inn(G))|=|Z(G)||G/Z(G)|=|G|,$$
because $Z(G)$ is elementary abelian by  Theorem 2.2 of  \cite{mac}. This completes the proof.
\end{proof}

\begin{cor}
Let $G$ be a finite Camina $p$-group. Then $|G|$ divides $|\Aut(G)|$. 
\end{cor}
\begin{proof} It is  a well known result \cite{dar} that nilpotence class of $G$ is at most 3. Also, it follows from \cite[Lemma 2.1, Theorem 5.2, Corollary 5.3]{mac} that $(G,Z(G))$ is a Camina pair. The result therefore follows from Theorem 2.2.
\end{proof}

An automorphism $\alpha$ of $G$ is called a class-preserving  automorphism of $G$ if $\alpha(x)\in x^G$ for each $x\in G$. The group of all class-preserving automorphisms of $G$ is denoted by $\Aut_c(G)$.  An automophism $\beta$ of $G$ is called a central automorphism if $x^{-1}\beta(x)\in Z(G)$ for each $x\in G$. It is easy to see that if $(G,Z(G))$ is a Camina pair, then the group of all central automorphisms fixing $Z(G)$ element-wise is contained in $\Aut_c(G)$.

\begin{rem}
{\em It follows from the proof of Theorem 2.2 that if $G$ is a finite $p$-group such that $(G,Z(G))$ is a Camina pair and $|Z(G)|\ge p^2$, then} 
$$|\Aut_c(G)|\ge|(\Aut_{M}^{Z(G)}(G))(\inn(G))|=|G|.$$
\end{rem}

Thus, in particular, we obtain the following result of Yadav \cite{yad1}.

\begin{cor}[{\cite[Corollary 4.4]{yad1}}]
Let $G$ be a finite $p$-group such that $(G, Z(G))$ is a Camina pair and $|Z(G)|\ge p^2$. Then $\Aut_c(G)/\inn(G)$ is non-trivial.
\end{cor}

The following example shows that Remark 2.4 is not true  if $|Z(G)|=p$.

\begin{expl}
{\em Consider a finite $p$-group $G$ of nilpotence class 2 such that $(G,Z(G))$ is a Camina pair and $|Z(G)|=p$.  Since $\cl(G)=2$, $\exp(G/Z(G))=\exp(G')$ and  hence $G'=Z(G)=\Phi(G)$. Let $|G|=p^n$, where $n\ge 3$, and let $\lbrace x_1, x_2, \ldots, x_{n-1}\rbrace$ be the minimal generating set of $G$.
Then }
$$|\Aut_c(G)|\le \prod_{i=1}^{n-1} |x{_i}^G|=p^{n-1}=|G/Z(G)|.$$
\end{expl}

\noindent {\bf Acknowledgment}:
Research of first author is supported by Thapar Institute of Engineering and Technology and also by SERB, DST grant no. MTR/2017/000581. Research of second author is supported by  SERB, DST grant no. EMR/2016/000019.

\

\end{document}